\documentclass[a4paper,12pt]{article}
\usepackage{mycommands}
\usepackage{amsfonts,amssymb,amsmath,times} 
\usepackage{url}

\makeatletter
\def\url@leostyle{%
  \@ifundefined{selectfont}{\def\UrlFont{\sf}}{\def\UrlFont{\small\ttfamily}}}
\makeatother
\def\abar{\bar{a}}
\def\bbar{\bar{b}}

\def\hbar{\bar{h}}

\def\vbar{\bar{v}}

\def\xbar{\bar{x}}
\def\ybar{\bar{y}}
\def\zbar{\bar{z}}

\def\Range{{\rm range}}

\def\phi{\varphi}
\def\FF{{\cal F}}
\def\RR{{\cal R}}

\def\OO{{\cal O}}
\def\C{\mathcal{C}}

\newcommand{\seq}[2][\omega]{\langle{#2}_i\rangle_{i\in {#1}}}
\newcommand{\D}[2]{\Delta_{{#1},{#2}}}
\newcommand{\bbd}[1]{\mathbb{#1}}
\title{Maximum VC families}
\author{Hunter Johnson{\footnote{hujohnson@jjay.cuny.edu}\hspace{.2cm}}
\\Department of Mathematics \& CS \\ John Jay College\\CUNY}

\begin{document}

\maketitle

\begin{abstract}
For a set $X$, a set system $\gc \sse 2^X$ with finite VC dimension is \textit{maximum} if it has the largest size allowable by Sauer's Lemma.  There is a natural association between set systems and parameterized formulas in first order model theory, where the set system associated to a formula is known as a ``definable family.''  Merging the two points of view, we may consider set systems which are maximum in the sense of computational learning theory, and stable in the sense of model theory.  We show that all stable maximum families are of the form $\gc \sse \{A \Delta B: A \in [X]^{\leq n}\}$ for some $n \in \omega$ and $B \sse X$.  We also examine maximum and non-maximum semi-algebraic families, give a model-completeness result generalizing the model completeness of $\langle \bbd{Q},<\rangle,$ and demonstrate that maximum families have UDTFS.  
\end{abstract}

\section{Introduction} \label{S:S1}
This paper is about set systems, in particular those arising as definable families in certain formal structures.  By ``family,'' we mean a collection of similar objects, such as circles, triangles, or other concepts, defined and parameterized by a single first-order formula.  The ``similarity'' of objects in a given family (or set system) is captured by its Vapnik-Chervonenkis (VC) dimension, defined below.  Only set systems with finite VC dimension (so-called VC classes) will be considered. 

In this paper we explore the model-theoretic properties of definable families which are maximum and/or maximal.  These terms refer to the largest set systems of a given VC dimension, where ``largest'' is measured in two distinct ways.  The maximal classes are apparently somewhat ad hoc, but maximum classes have a smoothness which gives strong structural properties.  These objects have long been of interest to researchers working in the field of Computational Learning Theory (CLT), where they have applications to PAC learning.\footnote{A reader wishing to review this literature might first read \cite{FlWa95} and then skim \cite{Fl89,BeLi98,KuWa07,RuRu09}. } %Maximum classes arise naturally in many geometric situations, particularly those involving arrangements of hyperplanes and half-spaces.  

%The main results of the paper are Theorem \ref{T:} and Proposition \ref{P:}.  Theorem \ref{T:} is a description of maximum classes which are stable in the sense of model theory, which we characterize as being essentially of the form $[X]^{\leq n}$ for a set $X$ and $n \in \omega$.  

The paper is organized as follows.  In Section \ref{S:S1} we give the main definitions, and explore the properties of set systems in the absence of any syntax or semantics.  Section \ref{S:S2} introduces model theory, and gives some of the main results, such as Proposition \ref{P:P4} and Theorem \ref{T:T4}.  Section \ref{S:S3} gives a theorem of Floyd, which shows that linearly parameterized semi-algebraic sets are maximum.  We then demonstrate that not all semi-algebraic families are as nicely behaved.  In Section \ref{S:S4} we give a model completeness condition related to the maximum property, which generalizes the model completeness of $\langle \bbd{Q},< \rangle$.  Section \ref{S:S5} defines UDTFS and shows its relation to the ``compression scheme'' notion from CLT.  We then illustrate how a compression scheme can be translated into a first-order condition, namely UDTFS.  In this section we also state some open problems, and prove the equivalence of two of them.  

Most stability-theoretic definitions are from Shelah \cite{Sh90}.  The Vapnik-Chervonenkis dimension was defined in \cite{VaCh71}.  A non-logician who wishes to read this paper might consult chapters 1,2 and 5 of \cite{Mar00}, and chapters 1 and 2 of \cite{Sh90}.

\subsection{}
A formula $\phi(v_1,v_2,\ldots,v_k)$ is said to be \textit{partitioned} if the free variables $v_1,\ldots,v_k$ are partitioned into parts $\xbar = v_{i_1},\ldots,v_{i_m} $ and $\ybar = v_{j_1}, \ldots,v_{j_{k-m}}$.  In this case we write $\phi(v_1,v_2,\ldots,v_k)=\phi(\xbar;\ybar)$.  Recall that as an ordinal, $n=\{0,1,\ldots,n-1\}$.

\begin{definition}
 For any partitioned formula $\phi(\xbar;\ybar)$, $n \in \omega$, and $\eta:n \rightarrow 2$, define $\phi_\eta(\ybar_1,\ldots,\ybar_n)=\exists \xbar  \bigwedge_{i\in n}\phi(\xbar;\ybar_i)^{\eta(i)}$, where $\phi(\xbar;\ybar)^1 := \phi(\xbar;\ybar)$ and $\phi(\xbar;\ybar)^0 := \neg \phi(\xbar;\ybar)$.
\end{definition}

\begin{definition}
 For any partitioned formula $\phi(\xbar;\ybar)$, define $\D{\phi}{n}=\{\phi_\eta(\ybar_1,\ldots,\ybar_n):\eta:n\rightarrow 2\}$.
\end{definition}

\begin{definition}
 Say that a sequence $\seq[I]{\abar}$ is $\D{\phi}{n}$-indiscernible if for any $i_1 < \cdots < i_n$ and $j_1 < \cdots < j_n$ from $I$, $tp_{\D{\phi}{n}}(\abar_{i_1},\ldots,\abar_{i_n})=tp_{\D{\phi}{n}}(\abar_{j_1},\ldots,\abar_{j_n})$.
\end{definition}

\begin{definition}
 Let $X$ a set, $A \sse X$, and $\mathcal{C} \sse 2^X$.  Define $\mathcal{C}(A) = \{c\cap A: c\in \mathcal{C}\}$.  Say that $\mathcal{C}$ \textit{shatters} $A$ if $\mathcal{C}(A)=2^A$.  Let the VC dimension of $\mathcal{C}$, denoted VC$(\mathcal{C})$, be defined as $\sup\{|A|: A \sse X, \mathcal{C} \text{ shatters } A\}$.
\end{definition}

For $n \in \omega$, $d \in \omega$, define $\Phi_d(n) = \sum_{i=0}^d {\binom{n}{i}}$.

%A result similar to the following was discovered independently in \cite{Sa72,Sh72} and \cite{VaCh71}.
\begin{lemma}[Sauer's Lemma \cite{Sa72,Sh72,VaCh71}]
 Suppose $\mathcal{C} \sse 2^X$ for a set $X$. If VC$(\mathcal{C}) = d$, and $A \sse X$ is finite, then $$|\mathcal{C}(A)| \leq \Phi_d(|A|)$$
\end{lemma}

\begin{definition}[\cite{We87}]
 Suppose $\mathcal{C} \sse 2^X$ and VC$(\mathcal{C})=d$.  Say that $\mathcal{C}$ is \textit{maximum} of VC-dimension $d$ (or $d$-maximum) if for all finite $A \sse X$, $$|\mathcal{C}(A)| = \Phi_d(|A|)$$
\end{definition}

%The notion of $\omega$-maximum families makes sense, and might be used to describe structures such as the Rado graph.  Proposition \ref{P:P7} is true if $d=\omega$ in this sense.

\begin{definition}[\cite{Du99}]
 Suppose $\mathcal{C} \sse 2^X$ and VC$(\mathcal{C})=d$.  Say that $\mathcal{C}$ is \textit{maximal} of VC-dimension $d$ (or $d$-maximal) if for any $c \in 2^X \setminus \mathcal{C}$, VC$(\mathcal{C} \cup \{c\}) = d+1$.
\end{definition}

\begin{example} $\mathcal{C} = [X]^{\leq d}$ for any $X$ with $|X|\geq d$ is $d$-maximal, where $[X]^{\leq d} := \{A \sse X: |A|\leq d\}$.  It is also $d$-maximum.
 
\end{example}

\begin{example}
 Let $X=\{0,1,2,3\}$, $d=2$, and build a suitable $\mathcal{C}$.  It is not hard to construct examples which are maximum and maximal, or maximal but not maximum.  
\end{example}

Any $\mathcal{C}$ is easily seen to have a maximal (though perhaps not a maximum) superclass $\ga \supseteq \mathcal{C}$, with VC$(\ga)=$VC($\mathcal{C}$), by Zorn's Lemma \cite{Du99}.

\begin{proposition} \label{P:P0}
Let $\gc \sse 2^X$ for a set $X$ and suppose $\mathcal{C}(A)$ is $d$-maximal for every finite $A \sse X$.  Let $\ga$ be any $d$-maximal superclass of $\gc$.  Then $\ga$ is the closure of $\mathcal{C}$ in the Tychonoff topology on $2^X$, and in particular is unique.% If $\mathcal{C}$ is maximum, then such an $\ga$ is unique.
\end{proposition}
\begin{proof}
 Let $X_0 \sse X$ finite and $f:X_0\rightarrow 2$.  We regard $2^X$ as ${}^X2$ and take sets of the form $\{g:X \rightarrow 2 \mid g \supseteq f\}$ as a basis for the topology.  Let $\bar{\gc}$ be the closure of $\gc$.  We must show $\bar{\gc} = \ga$.  Let $f \in \ga$.  For any finite $X_0 \sse X$, let $\OO_{f,X_0} = \{g \mid g \supseteq f\upharpoonright X_0\}$.  Since $\gc(X_0)$ is $d$-maximal, $f\upharpoonright X_0 \in \gc(X_0)$.  Then $\OO_{f,X_0} \cap \gc \neq \emptyset$.  This shows $f \in \bar{\gc}$.  It is easy to see that since $\ga$ is maximal, it must be closed.  Therefore $\ga=\bar{\gc}$.
\end{proof}
\begin{corollary}
 Let $\gc \sse 2^X$ for a set $X$ and suppose $\mathcal{C}$ is $d$-maximum.  Then $\bar{\gc}$ is the unique $d$-maximal superclass of $\gc$.
\end{corollary}

That $d$-maximum classes have unique $d$-maximal superclasses was first shown in Floyd's thesis \cite{Fl89}.  That maximal classes are closed in the Tychonoff topology was observed by Dudley \cite{Du99}.
\section{}  \label{S:S2}

Given a partitioned formula $\phi(\xbar;\ybar)$, let $\phi^*(\ybar;\xbar)=\phi(\xbar;\ybar)$.  For a monster model $\mcm$, model $\mcn$ and $B \sse M^{|\xbar|}$, let $\mathcal{C}_\phi(B)^\mcn=\{\phi(B,\bbar): \bbar \in N^{|\ybar|}\}$. We let $\mathcal{C}_\phi(B)$ where no model is specified implicitly denote $\mathcal{C}_\phi(B)^\mcm$.  The shorthand $\mathcal{C}_\phi(\mcm)$ will be used for $\mathcal{C}_\phi(M^{|\xbar|})^\mcm$.

\begin{definition}
 Say that a partitioned formula $\phi(\xbar;\ybar)$ is $d$-maximum (maximal) in $\mcm$ if $\mathcal{C}_\phi(\mcm)$ is $d$-maximum (maximal).  Say that a partitioned formula is $d$-*maximum (maximal) in $\mcm$ if $\mathcal{C}_{\phi^*}(\mcm)$ is $d$-maximum (maximal).
\end{definition}

Whether a formula is maximum depends only on the theory, since maximum-ness can be expressed as a set of first order sentences.  For maximal families, on the other hand, consider the $\{<\}$-formula $\phi(x;y,z,w)$ which expresses the relation
\begin{equation}
 \begin{cases}
  x < y &\text{ if } \hspace{.5cm}   y=z \wedge z<w \\
  x \leq y &\text{ if } \hspace{.5cm}   y=z \wedge z>w \\
  x=x &\text{ if }  \hspace{.5cm} y<z \\
  x\neq x &\text{ if }  \hspace{.5cm} y>z\\
 \end{cases}
\end{equation}

 We claim that this formula is 1-maximal in $\bbd{R}$ but not in $\bbd{Q}$.  Note that the formula encodes all left cuts $x<y$ including the ``limit cuts'' $x\leq y$, $x\neq x$, and $x=x$. The associated family is not maximal in $\bbd{Q}$ because $\gc_\phi(\bbd{Q})^\bbd{Q}$ does not include irrational cuts (which clearly do not increase the VC dimension).  It is maximal in $\R$, since that structure is Dedekind complete.  More precisely, $\gc_\phi(\R)^\R$ is closed in (ie. equals) $\gc_\phi(\R)$.  The claim then follows by Proposition \ref{P:P0}.

Since $\phi(x;y,z,w)$ is also maximum, this shows that maximum does not imply maximal on an infinite domain (see also \cite{Fl89,FlWa95}).  On any finite domain, however, maximum is easily seen to imply maximal.  If $\bbd{F}$ is a field and $k \sse \bbd{F}$ is a proper subfield, then the formula $\phi(x;y_1,\ldots,y_d) = ((x-y_1)(x-y_2)\cdots(x-y_d) = 0)$ gives $\mathcal{C}_\phi(k)^\bbd{F}$ $d$-maximum and $d$-maximal.  However $\mathcal{C}_\phi(\bbd{F})$ is only maximum (because it is ``missing'' the empty set).
%In an infinite field the formula $\phi(x;y_1,\ldots,y_d) \wedge \bigwedge_{i,j \in [d], i\neq j} y_i \neq y_j$ is only maximum.

\begin{example}
 Let $\phi(x;y_1,\ldots,y_{2m})$ be the $L=\{<\}$ formula $$\bigvee_{i \in [m]} (y_{2i-1} < x < y_{2i})$$ for some $m \in \omega$. Then in any infinite linear order, $\phi$ is maximum of VC dimension $2m$.\footnote{Unions of intervals have long been known to be maximum \cite{Fl89}.  An inductive proof is straightforward.}
\end{example}

  Let $\mathcal{C} \sse 2^X$ be $d$-maximum.  For any $A \sse X$ with $|A|=d+1$, $|\mathcal{C}(A)|=\Phi_d(d+1)=2^{d+1}-1$.  Let the unique $A^* \in 2^A \setminus \mathcal{C}(A)$ be called the \textit{forbidden label} for $\mathcal{C}$ on $A$ (Floyd's thesis, section 3.4).  

\begin{example}
 Let $X$ an infinite set, $d \in \omega$ and $\mathcal{C} = [X]^d$.  Then for any $A \sse X$ of cardinality $d+1$, the forbidden label for $\mathcal{C}$ on $A$ is $A^* = A$.
\end{example}

\begin{example}
 Let $X = \bbd{Q}$ and $\mathcal{C} = \mathcal{C}_{x<y}(\bbd{Q})$.  Then for $\{a,b\} \sse \bbd{Q}$ with $a < b$, the forbidden label for $\mathcal{C}$ on $\{a,b\}$ is $\{b\}$.
\end{example}

When $X$ has an ordering, a forbidden label can be naturally represented by a length $d+1$ binary string.  In the first example, above, the missing label can be viewed as $\overbrace{111\cdots11}^{d+1}$, and in the second as $01$.  We can connect forbidden labels to model theory as follows.

\begin{proposition} \label{P:P2}
 Suppose that $\phi(\xbar;\ybar)$ is a $d$-maximum formula.  Then for any $\abar_0,\ldots,\abar_{d} \in M^{|\ybar|}$, there is a unique $\eta^*:d+1 \rightarrow 2$ such that $\models \neg \phi_{\eta^*}(\abar_0,\ldots,\abar_{d}) \wedge \bigwedge_{\eta \neq \eta^*}\phi_{\eta}(\abar_0,\ldots,\abar_{d})$. Moreover $\{\abar_j : \eta^*(j)=1\}$ is the forbidden label of $\mathcal{C}_\phi(\mcm)$ on $\{\abar_0,\ldots,\abar_{d}\}$.
\end{proposition}
We will call $\eta^*$ in Proposition \ref{P:P2} the forbidden label of $\phi(\xbar;\ybar)$ on $\abar_0,\ldots,\abar_{d}$.

\begin{proposition}[Floyd \cite{Fl89}] \label{P:P3}
 Suppose $\mathcal{C} \sse 2^X$ is $d$-maximum and $d$-maximal.  For each $A \in [X]^{d+1}$ let $A^*$ denote the forbidden label for $\mathcal{C}$ on $A$. Then for any $c \in 2^X$, $c \in \mathcal{C} \iff \forall A \in [X]^{d+1} (c\cap A \neq A^*)$.
\end{proposition}
\begin{proof}
 Left to right is obvious.  Right to left follows from the fact that $\mathcal{C}$ is $d$-maximal -- any $c$ satisfying the right hand condition cannot increase the VC dimension of $\mathcal{C}$, and is therefore already in $\mathcal{C}$.  
\end{proof}

By Sauer's Lemma, the hypothesis on $\mathcal{C}$ in Proposition \ref{P:P3} will hold whenever $X$ is finite and $\mathcal{C}$ is $d$-maximum.

This has an interesting consequence for maximum formulas.  The following definition is a variation on classical NFCP from Keisler, which considers only positive instances.

\begin{definition}
 Say that a partitioned formula $\phi(\xbar;\ybar)$ is $n$-NFCP if for any $B \sse M^{|\ybar|}$, any set $\{\phi(\xbar;\abar)^{\eta(\abar)}: \abar \in B\}$ of $\pm \phi$-instances is consistent iff any $\Gamma \sse \{\phi(\xbar;\abar)^{\eta(\abar)}: \abar \in B\}$ with $|\Gamma| \leq n$ is consistent.
\end{definition}

\begin{proposition} \label{P:P4}
 Suppose the partitioned formula $\phi(\xbar;\ybar)$ is $d$-*maximum.  Then $\phi(\xbar;\ybar)$ is $(d+1)$-NFCP.
\end{proposition}
\begin{proof}
 Let $p(\xbar) = \{\phi(\xbar;\abar)^{\eta(\abar)} :\abar \in B\}$ be a set of $\pm \phi$-instances, and suppose it is $(d+1)$-consistent.  By compactness we may assume $B$ is finite. Let $pos(p) = \{\abar \in B: \eta(\abar)=1\}$. Note that since $B$ is finite, $\mathcal{C}_{\phi^*}(B)$ is maximal and maximum.  By Proposition \ref{P:P3}, $pos(p) \in \mathcal{C}_{\phi^*}(B)$, since it does not induce a forbidden label, by $(d+1)$-consistency.  But this implies that $p(\xbar)$ has a witness in $\mcm$, and so is consistent.
\end{proof}

 Proposition \ref{P:P4} can be juxtaposed with Helly's Theorem from combinatorial geometry (see \cite{Ma00}). 

%For any functions $f$, $g$, write $f \sqsubseteq g$ if $\dom(f) \subseteq \dom(g)$ and for all $x \in \dom(f)$, $f(x)=g(x)$.  
For any two functions $\eta:n \rightarrow 2$ and $\eta':m \rightarrow 2$ on natural numbers $m \leq n$ say that $\eta' \sqsubseteq \eta$ if there is an order preserving function $\nu:m \rightarrow n$ such that for all $i \in m$, $\eta'(i) = \eta(\nu(i))$.

\begin{theorem} \label{T:T1}
Let $\phi(\xbar;\ybar)$ be a $d$-*maximum formula.  Suppose $\seq[I]{\abar}$ is a sequence compatible with $\ybar$ which is $\D{\phi}{d+1}$-indiscernible.  Then $\seq[I]{\abar}$ is $\D{\phi}{\omega}$-indiscernible.  
\end{theorem}
\begin{proof}

Suppose $n \in \omega$, and that $i_1 < \cdots < i_n$, and $j_1 < \cdots, j_n$ are subsequences of $I$.  We must show $tp_{\D{\phi}{n}}(\abar_{i_1},\ldots,\abar_{i_n}) = tp_{\D{\phi}{n}}(\abar_{j_1},\ldots,\abar_{j_n})$, or equivalently that $\models \phi_\eta(\abar_{i_1},\ldots,\abar_{i_n})\equiv\phi_\eta(\abar_{j_1},\ldots,\abar_{j_n})$ for all $\eta: n \rightarrow 2$.  Let $\FF_1 = \{\eta:n \rightarrow 2: \models \phi_\eta(\abar_{i_1},\ldots,\abar_{i_n})\}$ and $\FF_2 = \{\eta:n \rightarrow 2: \models \phi_\eta(\abar_{j_1},\ldots,\abar_{j_n})\}$

Since $\seq[I]{\abar}$ is $\D{\phi}{d+1}$-indiscernible, all length $d+1$ subsequences have the same forbidden label $\eta^*$, where $\eta^*$ is as in Proposition \ref{P:P2}.  
By Proposition \ref{P:P3}, for any $\eta:n \rightarrow 2$, the following are equivalent.

\begin{enumerate}
\item $\eta \in \FF_1$
\item not $\eta^* \sqsubseteq \eta$.
\item $\eta \in \FF_2$
\end{enumerate}

Therefore $tp_{\D{\phi}{n}}(\abar_{i_1},\ldots,\abar_{i_n}) = tp_{\D{\phi}{n}}(\abar_{j_1},\ldots,\abar_{j_n})$, and $\seq[I]{\abar}$ is $\D{\phi}{\omega}$ indiscernible. 
\end{proof}

The following corollary to Theorem \ref{T:T1} says that maximum formulas allow a strong form of ``extraction" of indiscernibles.
\begin{corollary} \label{C:C1}
Let $\phi(\xbar;\ybar)$ be a $d$-*maximum formula and suppose $\seq[I]{\abar}$ is a sequence compatible with $\ybar$.  Then $\seq[I]{\abar}$ contains a $\D{\phi}{\omega}$-indiscernible subsequence.
\end{corollary}
\begin{proof}
By Ramsey's Theorem, $\seq[I]{\abar}$ contains a $\D{\phi}{d+1}$-indiscernible subsequence.  By Theorem \ref{T:T1}, this sequence is also $\D{\phi}{\omega}$-indiscernible.
\end{proof}

\begin{definition}[Shelah]
 A formula $\phi(\xbar;\ybar)$ is \textit{stable} in $\mcm$ if for some $N \in \omega$ there are not sequences $\abar_1,\ldots,\abar_N \in M^{|\xbar|}$ and $\bbar_1,\ldots,\bbar_N \in M^{|\ybar|}$ such that $\phi(\abar_i;\bbar_j)^\mcm \iff i < j$.
\end{definition}

%The condition on the right-hand side of the above definition is called the \textit{order property}.

\begin{lemma}[Shelah] \label{L:L1}
Suppose $\phi(\xbar;\ybar)$ is stable, and $\seq[I]{\abar}$ is an infinite $\D{\phi}{n}$-indiscernible sequence for $n \in \omega$.  Then $\seq[I]{\abar}$ is a $\D{\phi}{n}$-indiscernible set.
\end{lemma}
\begin{proof}
Suppose not.  Then since every permutation is a product of transpositions, for some subsequence $\abar_{i_0},\ldots,\abar_{i_l},\abar_{i_{l+1}},\ldots,\abar_{i_{n-1}}$ in $\seq[I]{\abar}$ and some $\eta:n \rightarrow 2$, $$\models \phi_{\eta}(\abar_{i_0},\ldots,\abar_{i_l},\abar_{i_{l+1}},\ldots,\abar_{i_{n-1}}) \wedge \neg \phi_{\eta}(\abar_{i_0},\ldots,\abar_{i_{l+1}},\abar_{i_l},\ldots,\abar_{i_{n-1}})$$  Suppose without loss that $\eta(l)=0$ and $\eta(l+1)=1$.  Define $$\psi(\xbar) = \bigwedge_{j\neq l,j\neq {l+1}}\phi(\xbar;\abar_{i_j})^{\eta(j)}$$ and $B = \psi(\mcm)$.  Note $$\phi_{\eta}(\abar_{i_0},\ldots,\abar_{i_l},\abar_{i_{l+1}},\ldots,\abar_{i_{n-1}}) \equiv \exists \xbar \in B \left ( \neg \phi(\xbar;\abar_{i_l}) \wedge \phi(\xbar;\abar_{i_{l+1}}) \right )$$Fix $N \in \omega$.  Since $\seq[I]{\abar}$ is a $\D{\phi}{n}$-indiscernible sequence, we may assume without loss that there are $\abar_j'$ in $\seq[I]{\abar}$, $j=1,2,\ldots,N$ such that $\abar_{i_l} < \abar_1' < \ldots < \abar_N' < \abar_{i_{l+1}}$.  Also since $\seq[I]{\abar}$ is a $\D{\phi}{n}$-indiscernible sequence, we have $$\models \exists \xbar \in B \left ( \neg \phi(\xbar;\abar_{i}') \wedge \phi(\xbar;\abar_{j}') \right ) \iff i < j$$ 
%This implies $\phi(B,\abar_N') \supsetneq \phi(B,\abar_{N-1}') \supsetneq \cdots \supsetneq \phi(B,\abar_1') \neq \emptyset$.
We can therefore find appropriate $\bbar_1,\ldots,\bbar_N$ in $B$ so that $\phi(\bbar_i;\abar_j')^\mcm \iff i < j$.  Since $N$ was arbitrary, $\phi(\xbar;\ybar)$ is unstable, a contradiction.
% It is now an easy exercise to show that $\phi(\xbar;\ybar)$ is not stable, a contradiction.
\end{proof}

In the following, we let $c_1 \Delta c_2$ denote the symmetric difference of sets $c_1$ and $c_2$.

\begin{lemma} \label{L:L1.5}
Let $\phi(\xbar;\ybar)$ be a stable $d$-*maximum formula.  Then there is a number $N \in \omega$ such that for all $c_1,c_2 \in \mathcal{C}_{\phi^*}(\mcm)$, $|c_1 \Delta c_2| < N$.
\end{lemma}
\begin{proof}
If $d=0$ then $|\mathcal{C}_{\phi^*}(\mcm)|=1$ and the lemma is trivial, so assume $d>1$. 

By compactness, it is enough to prove that for all $c_1,c_2 \in \mathcal{C}_{\phi^*}(\mcm)$, $|c_1 \Delta c_2| < \aleph_0$.  Suppose, by way of contradiction, that for some $c_1,c_2 \in \mathcal{C}_{\phi^*}(\mcm)$, $D=c_1 \Delta c_2$ is infinite.  Without loss, assume $D \sse c_1$.  By Corollary \ref{C:C1} and Lemma \ref{L:L1}, there is $D' \subseteq D$, an infinite $\D{\phi}{\omega}$-indiscernible set.  Let $\abar_0,\ldots, \abar_{d}$ be distinct elements in $D'$.  Consider $tp_{\D{\phi}{d+1}}(\abar_0,\ldots, \abar_{d})$.  By Proposition \ref{P:P2}, there is exactly one $\eta^*:d+1 \rightarrow 2$ such that $\models \neg \phi_{\eta^*}(\abar_0,\ldots, \abar_{d})$.  By choice of $D$, we know that $\eta^*$ is not a constant function. Suppose, without loss, that $\eta^*(0) \neq \eta^*(1)$.  Define $\mu:d+1 \rightarrow 2$ as
\begin{equation*}
\mu(i) =
\begin{cases}
\eta^*(i) & \text{ if } i \in (d+1) - \{0,1\} \\
1-\eta^*(i) & \text{ otherwise }
\end{cases}
\end{equation*}
Since $D'$ is an indiscernible set, the following are equivalent.
\begin{enumerate}
\item $\models \phi_{\eta^*}(\abar_0,\abar_1,\ldots, \abar_{d})$
\item $\models \phi_{\eta^*}(\abar_1,\abar_0,\ldots, \abar_{d})$
\item $\models \phi_{\mu}(\abar_0,\abar_1,\ldots, \abar_{d})$
\end{enumerate}
But then $\eta^*$ is not the unique forbidden label for $\{\abar_0,\ldots, \abar_{d}\}$, contradicting Proposition \ref{P:P2}.
\end{proof}
If $\mathcal{C} \sse 2^X$ for a set $X$ and $c \in \mathcal{C}$, define $\mathcal{C} \Delta c = \{f \Delta c : f \in \mathcal{C}\}$.  This operation clearly preserves many properties.

\begin{proposition}\label{P:P3.5}
 Let $\mathcal{C} \sse 2^X$ for a set $X$.  Then $\mathcal{C}$ is (maximum, maximal, stable) if and only if $\mathcal{C} \Delta A$ is (maximum, maximal, stable) for any $A \sse X$.
\end{proposition}

  If also $\ga \sse 2^X$ and $\mathcal{C} \Delta c = \ga$ then $\mathcal{C} = \ga \Delta c$, (because symmetric difference is associative and therefore $\mathcal{C} \Delta c\Delta c = \mathcal{C}$). Say that $\mathcal{C}$ is a subfamily of $\ga$ if $\mathcal{C} \sse \ga$. 
 
%The following essentially says that the only maximum stable formulas are the obvious ones. 
 
\begin{theorem} \label{T:T4}
A $d$-*maximum partitioned formula $\phi(\xbar;\ybar)$ is stable iff $\mathcal{C}_{\phi^*}(\mcm)$ is a subfamily of $[M^{|\ybar|}]^{\leq n} \Delta c$ for some $n \in \omega$ and any $c \in \mathcal{C}_{\phi^*}(\mcm)$.
\end{theorem}
\begin{proof}
 Let $\phi(\xbar;\ybar)$ be a $d$-*maximum partitioned formula.  Suppose $\phi(\xbar;\ybar)$ is stable. Then by Lemma \ref{L:L1.5}, for any $c \in \mathcal{C}_{\phi^*}(\mcm)$, $\mathcal{C}_{\phi^*}(\mcm)\Delta c \sse [\mcm^{|\ybar|}]^{\leq n}$ for some $n\in \omega$.  Therefore $\mathcal{C}_{\phi^*}(\mcm) \sse [\mcm^{|\ybar|}]^{\leq n}\Delta c$, by associativity of symmetric difference.  

Now, conversely, suppose $\mathcal{C}_{\phi^*}(\mcm) \sse [\mcm^{|\ybar|}]^{\leq n}\Delta c$.  By Proposition \ref{P:P3.5}, $[\mcm^{|\ybar|}]^{\leq n}\Delta c$ is a stable family of sets, and therefore $\mathcal{C}_{\phi^*}(\mcm)$ is also.  Then $\phi(\xbar;\ybar)$ is stable.
\end{proof}
In traditional notation, if $\gc$ is a maximum stable family on a set $X$, then $\gc \sse [X]^{\leq n} \Delta A$ for some $A \sse X$ and $n \in \omega$.
\section{} \label{S:S3}
 
Note that in any field, the formula $\phi(x;\ybar) = (p(x;\ybar)=0)$ where $p$ is a polynomial with coefficients $\ybar$ will be stable and maximum. On the other hand, polynomial equalities in dimensions greater than one will still be stable, but not maximum in general.\footnote{To see that polynomial equalities are stable formulas, note that they are quantifier free, and that every field is contained in its algebraic closure (which is stable).}  For instance, the symmetric difference of two distinct lines in $\bbd{R}^2$ is infinite.  Surprisingly, some semi-algebraic families are still in some sense approximately maximum.

For a set $H$ of real valued functions on a set $X$ and a real valued function $f_0(x)$ on $X$, let $f_0-H = \{f_0(x)-f(x):f \in H\}$.  Let $pos(f) = \{x \in X: f(x)>0\}$.  Define $pos(f_0-H)=\{pos(f_0(x)-f(x)): f \in H\}$.  

\begin{proposition}[Floyd \cite{Fl89}, Theorem 8.2] \label{P:P5}
 Let $H$ be a $n$-dimensional vector space of real valued functions on the set $X$, such that for every $X_0 \in [X]^n$, $H$ restricted to $X_0$ is also $n$-dimensional.  Further, for the real valued function $f_0(x)$ on $X$, assume that there are at most $n$ elements of $X$ such that $f_0(x)-f(x)=0$ for any $f \in H$. Then the class $\mathcal{C} = pos(f_0-H)$ is a maximum class of VC dimension $n$ on $X$.
\end{proposition}

Proposition \ref{P:P5} builds on the theorem by Dudley \cite{Du99} (Theorem 4.2.1) that families of the form $pos(f_0-H)$ have VC dimension $dim(H)$.
%Concretely, the space $\mathcal{C}$ of subsets in Proposition \ref{P:P5} can be taken as sets of positivity for a real polynomial with parameterized coefficients.  

\begin{example}[Floyd \cite{Fl89}, p. 104.] \label{E:E3}
 Let $X$ be a subset of $\R^2$, let $H$ be the three-dimensional vector space of functions of the form $f((x,y))=a_3y+a_2x+a_1$, and let $f_0((x,y))=-x^2-y^2$.  Then $f_0((x,y))-f((x,y))=-x^2-y^2-a_3y-a_2x-a_1$, and $pos(f_0-f)$ consists of all points for which $x^2+y^2+a_3y+a_2x+a_1 <0$.  These are the points contained in the circle with center $(-a_2/2,-a_3/2)$, and with radius $\sqrt{(a_3/2)^2+(a_2/2)^2-a_1}$.  Restrict $X$ to a subset of $\R^2$ such that $H$ is a 3-dimensional vector space on every subset of $X$ of cardinality 3.  This is satisfied if, for every three points $(x_1,y_1),(x_2,y_2),(x_3,y_3)$ in $X$, the three vectors $(x_1,x_2,x_3),(y_1,y_2,y_3),$ and $(1,1,1)$ are linearly independent.  Thus, $X$ cannot contain 3 collinear points.  Further restrict $X$ to a subset of $\R^2$ such that at most 3 points lie on the circumference of any circle.  Then $pos(f_0-H)$ is a 3-maximum class on $X$.
\end{example}

The $X$ in the above example can be taken as dense in an extension of $\R^2$. In particular, add countably many new constants $a_i$, paired into countably many $2$-tuples $\abar_i$ and let $\Gamma$ in the language of ordered rings (with constants from $\R$) express
\begin{enumerate}
 \item The $\abar_i$ are dense in the order topology.
 \item The $a_i$ satisfy the conditions described in Example \ref{E:E3}.
\end{enumerate}
Then this consistent set of sentences will give $X$ as desired.

The linear behavior of the parameters in Example \ref{E:E3} is important.  

\begin{proposition}
 Let $\phi(x_1,x_2;z_1,z_2,z_3) = x_1^2+x_2^2+z_3x_2+z_2x_1+z_1 < 0$ as in Example \ref{E:E3}.  There is no dense $Y \sse M^3$, for any $\R \preceq \mcm$ so that $\mathcal{C}_{\phi^*}(Y)$ is maximum.
\end{proposition}
\begin{proof}
 The reader may check that VC($\phi^*$)=2.  Suppose by way of contradiction that there is an extension $\mcm$ of $\R$ and a set $Y \sse M^3$, dense in the order topology, so that $\mathcal{C}_{\phi^*}(Y)$ is 2-maximum.  Let $A_1,A_2,A_3,$ and $A_4$ be subsets of the plane so that $A_j \cap A_i \cap A_l \neq \emptyset$ for any $i,j$, and $l$, but $A_1 \cap A_2 \cap A_3 \cap A_4 = \emptyset$.  We may assume without loss that for each $i=1,2,3,4$ there is a finite set $p_i(\xbar)$ of $\pm$-instances of $\phi(\xbar;\zbar)$ over $Y$ such that $A_i = p_i(\R^2)$. Then the set of formulas $\Sigma(\xbar) = p_1 \cup p_2 \cup p_3 \cup p_4$ is 3-consistent.  By Proposition \ref{P:P4}, $\Sigma(\xbar)$ is consistent.  Since $\Sigma$ is finite, and without loss finitely realizable in $\R$, it is realized in $\R$.  But this contradicts the fact that the $A_i$ have $A_1 \cap A_2 \cap A_3 \cap A_4= \emptyset$.
\end{proof}

Thus linearly parameterized semi-algebraic families are maximum on sets in ``general position,'' but those with non-linear parameterizations can strongly fail to be maximum.

Note that if we make a class $\gc$ from the topological ``frontiers'' of the objects described in Example \ref{E:E3}, then we get a stable maximum family of the type described in Theorem \ref{T:T4}.

\section{} \label{S:S4}
% \begin{proposition}
%  If $\ga \sse \mathcal{C}_{\phi^*}(\mcm)$ has the $(d+1)$-intersection property, then $\ga$ has the finite intersection property.  
% \end{proposition}
% \begin{proof}
% 
% 
% By Proposition \ref{P:P6} there exists a dense subset $X \sse \bbd{R}^{|\xbar|}$ such that $\mathcal{C}_{\phi}(X)$ is $d$-maximum.  Since $\mathcal{C}_{\phi^*}(\mcm)$ consists of open sets and polynomials are continuous, for any $\ga \sse \mathcal{C}_{\phi^*}(\mcm)$ and $\alpha \in \omega$, $\ga$ has then $\alpha$-intersection property if and only if $\ga' = \{\phi(\abar;\mcm) \in \ga: \abar \in X\}$ has the $\alpha$-intersection property.
% 
% Suppose $\ga \sse \mathcal{C}_{\phi^*}(\mcm)$ and $\ga$ has the $(d+1)$-intersection property.  Then $\ga'$ has the $(d+1)$-intersection property.  By Proposition \ref{P:P4}, $\ga'$ has the finite intersection property.  Therefore $\ga$ has the finite intersection property.
% 
% %By a well known theorem of Dudley, the VC dimension of $\phi(\xbar;\ybar)$ is $d:=|\ybar|$.  
% \end{proof}

Recall that a theory $T$ is model complete if whenever $\mcm,\mcn \models T$ and $\mcn \sse \mcm$, it follows $\mcn \preceq \mcm$.  The model completeness of $\langle \bbd{Q}, < \rangle$ can be generalized as follows.

\begin{proposition} \label{P:P7}

Fix $d \in \omega \cup \{\infty\}$. We will show that if $L = \{R(x,y)\}$ is a language with a single binary relation, and $T$ a $L$-theory with the axioms
\begin{enumerate}
 \item $R(x,y)$ $d$-maximum   
 \item $R(x,y)$ is symmetric (or antisymmetric)
 \item No finite $R(x,y)$-type is algebraic.  
\end{enumerate}
then $T$ is model complete.
\end{proposition}

\begin{proof}

We use Robinson's test.  Suppose $\mcn \sse \mcm$ and $\mcn,\mcm \models T$.  Consider the q.f. formula

$$\sigma(\vbar,v) = \bigwedge_{v_i \in \vbar}(v_i = v)^{\eta_1(i)} \wedge \bigwedge_{v_i \in \vbar}R(v_i,v)^{\eta_2(i)} \wedge \bigwedge_{v_i \in \vbar}R(v,v_i)^{\eta_3(i)}$$

where $\eta_j:|\vbar| \rightarrow 2$, for all $j \in \{1,2,3\}$.
\\

Let $\abar \in N^{|\vbar|}$, and suppose $\mcm \models \exists v \sigma(\abar,v)$.  We must show $\mcn \models \sigma(\abar,b)$ for some $b \in N$.  Without loss of generality, we may assume that $\eta_1$ is the zero function.  We first show that 

$$\mcn \models \exists v \bigwedge_{a_i \in \abar}R(a_i,v)^{\eta_2(i)} \wedge \bigwedge_{a_i \in \vbar}R(v,a_i)^{\eta_3(i)}$$

Since $R(x,y)$ is assumed to be symmetric (or antisymmetric) and $\sigma$ is consistent, the above statement holds if and only if 

\begin{equation}
\mcn \models \exists v \bigwedge_{a_i \in \vbar}R(v,a_i)^{\eta_2(i)} \label{eqn:*L}
\end{equation}

But this last statement must hold.  For if not,

$$|\mathcal{C}_R(range(\abar))^\mcn| < |\mathcal{C}_R(range(\abar))^\mcm|$$

%\Sigma_{i \leq d} {{|\abar|}\choose{i}}$$

contradicting that $R$ is maximum in $\mcn$.

By condition 3 and equation (1), $\mcn$ has a witness $b$ to $\exists v \bigwedge_{a_i \in \abar}R(v,a_i)^{\eta_2(i)}$ such that 

$$\mcn \models \bigwedge_{a_i \in \vbar}(a_i \neq b)$$

Thus

$$\mcn \models \sigma(\abar,b)$$

and consequently $T$ is model complete.  
\end{proof}
If (2) in Proposition \ref{P:P7} is replaced with the assumption that $R(x,y)$ is 2-sorted, or that $x$ and $y$ are otherwise incompatible, then an analog of the proposition goes through, if we further assume that $R^*(y,x)$ is maximum.
% 
% \begin{proposition} The theory in \ref{P:P7} is consistent
% \end{proposition}
% \begin{proof}
%  Clear if $d=1$, since $\langle \bbd{Q}, < \rangle$ is a model.  If $d>1$ ??????
% \end{proof}

\section{UDTFS} \label{S:S5}

If $B \sse M^{|\ybar|}$ and $\phi(\xbar;\ybar)$ is a partitioned formula, a complete $\phi$-type over $B$ is any consistent set of formulas $$p(\xbar)=\{\phi(\xbar;\bbar)^{\eta(\bbar)}: \bbar \in B\}$$
for some $\eta:B \rightarrow 2$.  We let $S_\phi(B)$ represent the set of all complete $\phi$-types over $B$.

\begin{definition}
  Let $\phi(\xbar,\ybar)$ a partitioned formula.  We say that $\phi$ has uniformly definable types over finite sets (UDTFS) if for some $N \in~\omega$ there exists a set of formulas $\{\psi_l(\ybar,\ybar_0,\ldots, \ybar_{n-1}):l \in N\}$ such that for any any finite $B \sse M^{|\ybar|}$ and any $p \in S_{\phi}(B)$ there is $l \in N$ and $\bbar_0,\ldots, \bbar_{n-1} \in B$ such that for all $\bbar \in B$,
$$ \phi(\xbar,\bbar) \in p \iff \models \psi_l(\bbar,\bbar_0,\ldots, \bbar_{n-1}).$$
\end{definition}

The definition of UDTFS was based on the notion of a compression scheme from computational learning theory \cite{JoLa10}. 
Warmuth and Littlestone \cite{LiWa86} say that
$\C$ admits a {\em $d$-dimensional compression\/} if, given
any finite subset $F$ of $X$, and any set $A\in\C$,
there is a $d$-element subset $S$ of $F$
such that the set $A\cap F$ can be recovered from
the sets $S\cap A$ and $S\setminus A$.

\begin{example}[Warmuth \& Littlestone] \label{E:E5}
 Let $\gc$ be the set of all solid axis-parallel rectangles in the plane and $F$ a finite set of points.  Fix a rectangle $R \in \gc$.  Let $S$ be the topmost,leftmost,rightmost and lowest points in $F \cap R$.  Let $\tilde{R} = \bigcap\{R' \in \gc: S \sse R'\}$.  Then $\tilde{R}\cap F=R\cap F$.
\end{example}
Note that in Example \ref{E:E5}, $S \setminus R$ was not needed, and it is not necessary that $\tilde{R} \in \gc$.

A more technically useful tool is an extended compression scheme, also due to Warmuth and Littlestone, and defined as follows.

 In the following, we identify $\mathcal{C} \sse 2^X$ with $\{f_c:X \rightarrow 2: c \in \mathcal{C}\}$, where $f_c(x)=1 \iff x \in c$.  Furthermore, a function is identified with its graph, so that $f_1 \sse f_2$ iff $f_1$ is a restriction of $f_2$.

For $B\subseteq X$, the notation $\mathcal{C}|_B$ denotes
the set of restrictions $\{f|_B:f\in\mathcal{C}\}$
and 
$$\mathcal{C}|_{\rm fin}=\bigcup\{\mathcal{C}|_B:\hbox{$B$ a finite subset of $X$ with
 $|B|\ge 2$}\}$$

\begin{definition}
 Fix $\mathcal{C} \sse {}^X\{0,1\}$.  $\mathcal{C}$ is said to have an {\em extended $d$-compression\/} 
if there is a
{\em compression function\/} $\kappa: \mathcal{C}|_{\rm fin} \ra [X]^{\leq d}$
and a finite set $\RR$ of {\em reconstruction functions} 
$\rho: [X]^{\leq d}\ra {}^X\{0,1\}$ such that for every $f\in\mathcal{C}|_{\rm fin}$
\begin{enumerate}
 \item $\kappa(f)\subseteq \dom(f)$
 \item $f\subseteq \rho(\kappa(f))$ for at least one $\rho\in\RR$.
\end{enumerate}
We say that $\mathcal{C}$ has an {\em extended $d$-sequence compression\/} if there
there is a compression function $\kappa: \mathcal{C}|_{\rm fin} \ra X^d$
and a finite set $\RR$ of {\em reconstruction functions} 
$\rho: X^d\ra {}^X\{0,1\}$ such that for every $f\in\mathcal{C}|_{\rm fin}$,
$\Range(\kappa(f))\subseteq\dom(f)$, and $f\subseteq \rho(\kappa(f))$ for at least one $\rho\in\RR$.

\end{definition}

The existence of either of these $d$-compressions is equivalent.  See \cite{JoLa10} for proofs.

%The following is one of the most important results in the theory of compression schemes.

\begin{theorem}[\cite{FlWa95}, Theorem 11]
 Suppose $\mathcal{C} \sse 2^X$ for a set $X$ is $d$-maximum.  Then $\mathcal{C}$ has a $d$-dimensional compression.
\end{theorem}

This was later improved to an extended $d$-compression with $|\RR| = 1$ in \cite{KuWa07}.  An extended $d$-compression with $|\RR|=1$ is usually called an ``unlabeled'' $d$-compression scheme.  While the improved result does not translate to a first-order statement (ie UDTFS) the original result does, as we now show.  

\begin{definition}
 Let $\phi(\xbar;\ybar)$ a partitioned formula and $B \sse M^{|\ybar|}$ finite. Say that $p \in S_\phi(B)$ internally shatters $A \sse B$ if $p \upharpoonright_{B \setminus A}$ has $2^{|A|}$ extensions to $B$.
\end{definition}

Define the independence dimension of $\phi(\xbar;\ybar)$ as VC($\phi^*)$, and say that $A \sse M^{|\ybar|}$ is independent if $A$ is shattered by $\phi^*(\ybar;\xbar)$.  

\begin{lemma} \label{L:L2}
 Suppose $\phi(\xbar;\ybar)$ is a partitioned formula with independence dimension $d$.  Suppose further that for any finite $B \sse M^{|\ybar|}$ and $p \in S_\phi(B)$, $p$ internally shatters some $A \in [B]^d$.  Then $\phi(\xbar;\ybar)$ is UDTFS.
\end{lemma}
\begin{proof}
 Let $B$, $p$ and $A$ be given. Let $\bbar \in B \setminus A$ and consider $\phi(\xbar;\bbar)^t$.  Since $|A|=d$, there must be exactly one value for $t \in 2$ for which $\phi(\xbar;\bbar)^t$ is consistent with any $q \in S_\phi(A)$, namely the value for which $\phi(\xbar;\bbar)^t \in p$.  Otherwise $A \cup \{\bbar\}$ would be independent, contradicting $Idim(\phi)=d$.  
Define $$\theta(\ybar;\ybar_0,\ldots,\ybar_{d-1}) = \bigwedge_{\eta:d \rightarrow 2} \exists \xbar \left ( \phi(\xbar;\ybar) \wedge  \bigwedge_{i \in d} \phi(\xbar;\ybar_i)^{\eta(i)} \right )$$ 

For $\eta: d \rightarrow 2$, define 

$$\psi_\eta(\ybar;\ybar_0,\ldots,\ybar_{d-1}) = \bigwedge_{i \in d}\left ( \ybar = \ybar_i  \rightarrow (\ybar=\ybar)^{\eta(i)} \right ) \wedge \left ( \bigwedge_{i \in d}  \ybar \neq \ybar_i \right ) \rightarrow \theta(\ybar;\ybar_0,\ldots,\ybar_{d-1})$$

Now for any $p \in S_\phi(B)$, if $\abar_1,\ldots,\abar_d$ is internally shattered by $p$, then for some  $\eta: d \rightarrow 2$ and all $\bbar \in B$,
$$ \phi(\xbar,\bbar) \in p \iff \models \psi_\eta(\bbar,\abar_1,\ldots, \abar_{d}).$$

The formulas $\psi_\eta(\ybar;\ybar_0,\ldots,\ybar_{d-1})$ suffice.
\end{proof}

The fact that all $d$-*maximum formulas satisfy the hypothesis in Lemma \ref{L:L2} is attributed to Emo Welzl \cite{We87}.  See also Theorem 10 in \cite{FlWa95}.

\begin{lemma}[Welzl, 1987] \label{L:L3}
 Suppose $\phi(\xbar;\ybar)$ is a $d$-*maximum partitioned formula. Then for any finite $B \sse M^{|\ybar|}$ and $p \in S_\phi(B)$, $p$ internally shatters some $A \in [B]^d$. 
\end{lemma}

\begin{theorem}
 Suppose $\phi(\xbar;\ybar)$ is a $d$-*maximum partitioned formula.  Then $\phi(\xbar;\ybar)$ is UDTFS.
\end{theorem}
\begin{proof}
 By Lemmas \ref{L:L2} and \ref{L:L3}.
\end{proof}

The following are open questions regarding UDTFS.  In each case the left to right direction is known to hold. A $L$-theory $T$ has UDTFS if every $L$-formula has UDTFS.
\begin{enumerate}
 \item  $\phi(\xbar;\ybar)$ is UDTFS iff $\mathcal{C}_{\phi^*}(\mcm)$ has an extended $d$-compression scheme for some $d$.
 \item Every reduct of a theory $T$ has UDTFS iff $T$ has UDTFS.
 \item $\phi(\xbar;\ybar)$ has UDTFS iff $\phi(\xbar;\ybar)$ has finite VC dimension.  
 \item $\mathcal{C} \sse 2^X$ has a compression scheme of order $d$ iff $\mathcal{C}$ has finite VC dimension $d$.  
%  \item Suppose $\phi(\xbar;\ybar)$ is such that $\mathcal{C}_{\phi^*}(\mcm)$ has an extended $d$-compression scheme.  Does $\phi(\xbar;\ybar)$ have UDTFS?
%  \item Suppose $T$ has UDTFS (ie every formula is UDTFS) in a language $L$.  Does any every reduct of $T$ have UDTFS?
%  \item Suppose $\phi(\xbar;\ybar)$ has finite VC dimension.  Is $\phi(\xbar;\ybar)$ UDTFS?
%  \item Suppose $\mathcal{C} \sse 2^X$ has finite VC dimension $d$.  Does $\mathcal{C}$ have a compression scheme of order $d$?  

\end{enumerate}

Question (3) is the NIP $\iff$ UDTFS conjecture discussed in \cite{JoLa10, Gu11}.  Question (4) is one of the principal unsolved problems in computational learning theory \cite{Wa03}. % The mention of the asymptotic ``order'' of the compression dimension in (4) relates to the hope the problem may be solved by embedding all $d$-maximal classes into a maximum class whose VC dimension is not too much larger.  

\begin{proposition}
 (1) $\iff$ (2)
\end{proposition}
\begin{proof}
 Let $\phi(\xbar;\ybar)$ be given and suppose (1) holds, viz that every $\phi(\xbar;\ybar)$ for which $\mathcal{C}_{\phi^*}(\mcm)$ has an extended compression scheme is UDTFS.  Suppose $T$ is an $L$-theory in which $\phi(\xbar;\ybar)$ is UDTFS.  Let $T'$ be a reduct of $T$ to $L' \sse L$ such that $\phi \in L'$.  We will show that $\phi(\xbar;\ybar)$ is UDTFS in $T'$.  It is easy to see that $\phi(\xbar;\ybar)$ UDTFS in $T$ implies that when $\mcm \models T$, $\mathcal{C}_{\phi^*}(\mcm)$ has an extended compression scheme.  Let $\mcm'$ be an $L'$ reduct of $\mcm$.  Then $\mathcal{C}_{\phi^*}(\mcm) = \mathcal{C}_{\phi^*}(\mcm')$, and so $\mathcal{C}_{\phi^*}(\mcm')$ has an extended compression scheme.  But then by (1), $\phi(\xbar;\ybar)$ is UDTFS in $T'$.

Now conversely suppose that (2) holds, viz that if $T$ is UDTFS in $L$ then the $L'$ reduct $T'$ is UDTFS.  Let $\phi(\xbar;\ybar)$ be given and suppose $\mathcal{C}_{\phi^*}(\mcm)$ has an extended compression scheme for an $L$ model $\mcm$ of $T$.  We must show that $\phi(\xbar;\ybar)$ is UDTFS in $T$.  Without loss of generality, $\mathcal{C}_{\phi^*}(\mcm)$ has an extended sequence compression scheme, with reconstruction functions $\RR$.  For each $\rho(\ybar_0,\ldots,\ybar_{d-1}) \in \RR$, add a predicate $\rho'(\ybar;\ybar_0,\ldots,\ybar_{d-1})$ to $L$ to get a language $L^\RR$.  Let $\mcm^\RR$ be an $L^\RR$ expansion of $\mcm$ such that for all $\abar,\abar_0,\ldots,\abar_{d-1}\in {M^\RR}^{|\ybar|}$, $$\rho'(\abar,\abar_0,\ldots,\abar_{d-1})^{\mcm^\RR} \iff \rho(\abar_0,\ldots,\abar_{d-1})(\abar)=1$$Then $\phi(\xbar;\ybar)$ is UDTFS in $Th(\mcm^\RR)$.  By (2), it is also UDTFS in the $L$ reduct, $T$.
\end{proof}

Clearly if either (1) or (2) holds, then (3) and (4) are qualitatively the same question.

\section{Conclusions}
%In the work above, we have investigated stable maximum formulas

% It remains to give a similar characterization in the unstable case.  Much work has gone into analyzing maximum classes which can be realized as arrangements of hyperplanes (cite cite).   Can the general space of $d$-maximum families be somehow classified or described?

The following are typical examples of maximum VC classes:  
\begin{enumerate}
 \item  $[X]^{\leq n}$
 \item Unions of boundedly many intervals in $\R$
 \item Sets of positivity for a finite dimensional real vector space of real valued functions, restricted to points in general position.
\end{enumerate}

Note that the second example is a special case of the third.  We have shown that all stable maximum families are essentially of the first type.  It is unclear whether there are similar ``universal'' set systems for unstable maximum families, or if the above unstable examples are  the only possibilities.  Some work on this problem is done in \cite{BeLi98}.  Results on geometric characterizations of maximum families can be found in \cite{RuRu09}, as well as an algorithm for generating all finite maximum families. 

It has been remarked that there is a curious absence of natural examples of maximal but not maximum classes \cite{Fl89}.  At the same time, it seems that ``most'' of the wild (ie. random) maximal classes are not maximum \cite{Fl89,FlWa95}. 

On this same topic, the property of being maximal (in particular, closed in the Tychonoff topology) implies a strong condition on the type space which will not always be possible.  A model $\mcn$ with a binary relation $R$ whose type space $\gc_R(\mcn)^\mcn$ is closed in (ie equal to) $\gc_R(\mcn)$ must be more than saturated--it must realize all types over itself.  Therefore, no equivalence relation with infinitely many classes will ever be closed (because of the ``not equivalent to anything'' type). In fact a Dedekind complete order is a rare example of something natural which is maximal. The notion of a closed relation is similar to algebraic compactness in the theory of modules, where a pure injective module must realize all positive primitive formulas over itself.  That condition is relatively easy to realize, however, since negative instances are not considered.

%A longstanding question is whether all set systems of VC dimension $d$ can be embedded in maximum concept classes of size $d+k$ for some uniform constant $k$.  If this is true, then assertion (3) in Section \ref{S:S5} holds, and the ``sample compression conjecture'' from CLT would be true.  There is some evidence that this may be false, namely that there are systems of dimension $d$ which do not embed into certain well-understood maximum classes of dimension $d+k$ cite.

If a definable family embeds in a maximum class, it inherits some good properties, such as the existence of a compression scheme.  In dimension 1 in any (weakly) o-minimal or strongly minimal theory, all definable families are sub-families of maximum definable families.  We may ask how many of the good properties of such theories are related to this fact. 

%We hope we have shown that the properties of being maximum and maximal are model theoretically interesting, and have some applications.

   %In particular such a family will have a compression scheme of some finite size.  Thus one way of resolving the UDTFS $\iff$ NIP conjecture would be to show that any family can embed in a maximum family, and then establish assertion (1) in Section \ref{S:S5}.    

\bibliographystyle{plain}
%\bibliography{/home/hunter/Documents/Tex/refs.bib}
 \bibliography{refs}

\end{document}